\def\ZZ         {{\mathbb Z}}

\def\CC         {{\mathbb C}}

\def\QQ         {{\mathbb Q}}
\def\PP         {{\mathbb P}}
\def\NN         {{\mathbb N}}
\def\ZZ         {{\mathbb Z}}
\def\X          {{\cal X}}

\def\F          {{\cal F}}
\def\G          {{\cal G}}
\def\H         {{\cal H}}

\def\L          {{\cal L}}

\def\O          {{\cal O}}

\def\V           {{\cal V}}
\def\ii         {{\rm i}}
\def\ee         {{\rm e}}

\def\Char       {{\rm Char}}
\def\Ker        {{\rm {Ker}}}

\def\dim        {{\rm dim}}

\def\Ell        {{\cal ELL}}

\def\cal        {\mathcal}

\documentclass{amsart}
\usepackage{amssymb}
\usepackage{verbatim}
\usepackage{eucal}
\usepackage{eufrak}
\newtheorem{theorem}{Theorem}[section]
\newtheorem{lemma}[theorem]{Lemma}
\newtheorem{prop}[theorem]{Proposition}
\newtheorem{corollary}[theorem]{Corollary}

\newtheorem{dfn}[theorem]{Definition}

\theoremstyle{remark}
\newtheorem{remark}[theorem]{Remark}

\newtheorem{example}[theorem]{Example}

\title{Elliptic genus of phases of $N=2$ theories}

\author{Anatoly Libgober}
\address{Department of Mathematics\\
University of Illinois\\
Chicago, IL 60607}
\email{libgober@math.uic.edu}

\thanks{Author supported by a grant from Simons Foundation}
\begin{document}
\begin{abstract} We discuss an algebro-geometric description 
of Witten's phases of N=2 theories and propose a definition 
of their elliptic genus provided some conditions on singularities
of the phases are met. For Landau-Ginzburg phase
one recovers elliptic genus of LG models proposed in physics literature
in early 90s.  
For certain transitions between phases  we derive invariance 
of elliptic genus from an 
equivariant form of McKay correspondence for elliptic genus. 
As special cases one obtains Landau-Giznburg/Calabi-Yau correspondence
for elliptic genus of weighted homogeneous potentials  
as well as certain hybrid/CY correspondences. 
\end{abstract}

\maketitle

\section{Introduction}

Elliptic genera appeared in mathematics and physics literature 
in the middle 80s as invariants associated 
with manifolds (cf. \cite{landweber}). Around the same time it was realized 
that elliptic genus can be associated with superconformal 
field theories which depend on rather different type of data 
e.g.minimal models, Landau-Ginzburg models etc. (cf. \cite{wittenlg},
\cite{kawai}). In \cite{n=2}, Witten proposed geometric 
procedure based on use of
variations of symplectic quotients for specific actions of Lie groups, 
which relates Landau Ginzburg models 
to the sigma-models
and hence to the manifolds. In fact Witten's construction lead not just to 
Landau-Ginzburg or sigma models but to a host of others, which 
he called {\it phases of $N=2$ theories}, and which, besides just mentioned
types, include hybrid models, gauged Landau Ginzburg models etc.
The purpose of this note is to associate elliptic genus to 
(a generalization of) Witten's phases of $N=2$ theories.

The key constraint for such definition it seems 
should be invariance of elliptic 
genus in transitions between $N=2$ supersymmetric phases. The cases when 
a relation between Landau-Ginzburg and sigma-models are involved 
called Landau-Ginzburg/Calabi Yau correspondence. 
Study of such correspondence in the various contexts
was the subject of enormous number of beautiful works 
over last 20 years of which 
we shall mention only a few.
In the context of A-and B-models LG/CY correspondence was considered
by Chiodo, Iritani and Ruan  (cf. \cite{chiodo} and further references there)
in particular yields relation between Gromov-Witten and FJRW's theories.
In the context of homological mirror symmetry, the LG/CY correspondence
was obtained by Orlov (cf. \cite{orlov})  
based on Kontsevich definition of categories 
associated with singularities and was recently substantially extended 
by Ballard, Favero and Katzarkov (cf. \cite{katzarkov})
(cf. also \cite{halpern}).
See a review and discussion of some of these 
and numerous other recent developments in 
\cite{katzarkovpr}, \cite{katzarkov1}.
In the case of elliptic genus LG/CY correspondence 
for homogeneous polynomials was obtained by 
Gorbounov and Malikov (in fact this work deals 
with the vertex algebras associated with LG/CY data,
cf. \cite{GM}, also \cite{gorboch}). Extension of LG/CY 
correspondence for vertex algebras of (0,2)-toric models 
was considered in \cite{borisovkaufmann}.

Our generalization of Witten phases are Geometric Invariant Theory  (GIT)
quotients of the total spaces of a $G$-equivariant line bundle 
($G$ is a reductive group) on a quasi-projective manifold (cf. 
Def. \ref{phasesdef}). 
Fiberwise $\CC^*$-action on the total space of such bundles,
can be used to define the action on the GIT quotient
(cf. Prop. \ref{actionexistence}).
We work under assumptions that singularities of  
GIT quotients are  not too bad with precise conditions 
discussed in section \ref{actionssection}.
For each compact component of the fixed point set one 
has a contribution into equivariant elliptic genus 
discussed in section \ref{ellgensection}. 
This contribution is a holomorphic function 
on $\CC \times \CC \times \H$ where $\H$ is the upper half plane. 
It depends on the equivariant Chern classes of the 
neighborhood of this component in the GIT quotient
(cf. section \ref{ellgensection}). 
The elliptic genus of our phase is defined as the function $E(z,\tau)$
which is the restriction of the contribution of the fixed point set of 
the $\CC^*$-action into the equivariant elliptic genus of the phase
to $\Delta \times \H \subset \CC\times \CC \times \H$
where $\Delta$ is the diagonal of $\CC \times \CC$.
We show that such restriction yields, under 
certain Calabi Yau type conditions, 
a function on $\CC\times \H$ which is Jacobi form of expected weight 
and index (cf. Section \ref{modularitysec}).
In the case when the total space of the 
line bundle is $\CC \times \CC^n$
with the action of $\CC^*$ given by $\lambda(p,x_1,...,x_n)=
(\lambda^{-D}p,\lambda^{w_1}z_1,....,\lambda^{w_n}z_n)$ there are two GIT quotients
specified by a choice of linearization of this action (cf. example 
\ref{weightedaction}).
One of the GIT quotients is the quotient of $\CC^n$ 
by the action of the cyclic group 
and our elliptic genus $E(z,\tau)$ is just the elliptic genus of 
Landau Ginzburg model appearing on physics literature 
(cf. \cite{aharony},\cite{berglund}, \cite{kawai}).
Moreover, one readily identified the elliptic genus 
of the second GIT quotient with the orbifold elliptic 
genus of hypersurface in weighted projective space.
The identity 
between elliptic genera corresponding 
to different linearization 
represents a wall crossing phenomenon and can be derived 
in many cases from the equivariant version of McKay correspondence 
(cf. \cite{waelder})
for elliptic genus obtained in \cite{BLAnnals}.
In particular, the equality of elliptic genera 
corresponding to hybrid and other 
models lead to interesting new identities between Jacobi forms and
several explicite examples are given in the section \ref{maintheoremsection}.
Paper is concluded with description of possible firther developments.

The point of view on Witten phases described in this paper is close to the one 
taken in \cite{katzarkov} but the difference is 
that in the later work variation of GIT quotient occurs in the 
base of equivariant bundle while we consider variation 
GIT quotients of the total space of 
line $G$-bundles\footnote{the work \cite{katzarkov} uses differently the term 
{\it ``gauged 
Landau Ginzburg model''}. We use it (cf. definition 
\ref{phasesdef} (iii))
essentially
in the way it was used in Witten paper \cite{n=2}.} 
The piece of the data of LG model 
consisting of the section of $L$ in \cite{katzarkov}
in this note is replaced
by the action of the group $G$ on the total space of $L$.
The potential is a $G$-equivariant section 
of $L$, i.e. in the case of trivial bundle on $\CC^n$
(classical LG model) is essentailly a weighted homogeneous polynomial
(cf. Example \ref{weightedaction}). It 
does not play direct role in our definition of the elliptic genus 
of a phase. 

Finally, I thank Max Planck Institute for support of my visit where 
the work on this paper was concluded.
I also want to thank Nitu Kitchloo for useful comments on content of this 
paper (during conference on Modular Forms in Topology and Analysis 
where it was reported).

\section{GIT  quotients of total spaces of line bundles.}

\subsection{Linearizations of actions of reductive groups and corresponding GIT quotients.}

Let $X$ be a smooth quasi-projective variety acted upon by 
a reductive group $G$.
A linearization $\kappa$ of a line bundle 
$L$ (cf. \cite{MFK},\cite{dol}) is 
a fiberwise linear action of $G$ on the total space $\L$ of a line bundle 
$L$ commuting 
with projection of $\L$ on $X$.
Denote by $Pic^G(X)$ the  
group of linearized $G$-bundles. 
One has the exact sequence:
(cf. \cite{KKV}, \cite{dol}):
\begin{equation}\label{linearizations}
   0 \rightarrow H^1(G,\O(X)^*) \rightarrow Pic^G(X)
\rightarrow Pic(X)^G \subset Pic(X)
\end{equation}
where $Pic(X)^G$ is the group of $G$-invariant line bundles.

If $X$ is an affine space or is projective and $G$ 
is connected then 
\begin{equation}\label{algchar}
H^1(G,\O(X)^*)=\Char(G)
\end{equation}
With the data $X,G,L,\kappa,(L \in Pic^G(X))$ as above, 
one associates subsets of $X$, the semi-stable, 
stable and unstable 
loci, $X^{ss}_{\kappa},X^s_{\kappa},X^{us}_{\kappa}$ respectively, 
as follows (cf. \cite{MFK}):
\begin{equation}\label{defstability}
  X^{ss}_{\kappa}=\{x \in X \vert \exists m>0 \ {\rm and} \ s \in \Gamma(X,L^{\otimes m})^G, 
s(x) \ne 0, \ {\rm such} \ {\rm that} \  X_s \ {\rm is} \ {\rm affine} \}
\end{equation}
\begin{equation} \ \ \ \ \ \ \ \  X^s_{\kappa}=\{ x \in X \vert \ {\rm as} \ {\rm in} \ {(\ref{defstability})} \ 
{\rm and} \ G_x \ {\rm is} \ {\rm finite} \ {\rm  and} \ \overline {Gx}=Gx 
\ {\rm in} \ X_s \} 
\end{equation}
\begin{equation}
 \ \ \ \ \ \ \ \ \ \ \ \ \ X^{us}_{\kappa}=X\setminus X^{ss}_{\kappa}
\end{equation}
(here $X_s=\{y \in X \vert s(y)\ne 0\}$, $G_x$ is the stabilizer of $x$
and $\overline {Gx}$ is the closure of $G$-orbit of $x$). 
\begin{dfn}
GIT 
quotient $X//_{\kappa}G$ of $X$ corresponding to linearization $\kappa$ 
is the categorical $G$-quotient
of $X^{ss}_{\kappa}$ i.e. a morphism $p: X^{ss}_{\kappa} \rightarrow 
X^{ss}_{\kappa}//G$ such that 
for any $G$-compatible morphism $X^{ss}_{\kappa} \rightarrow Z$ there 
exist unique factorization
$X^{ss}_{\kappa} \rightarrow X^{ss}_{\kappa}//G \rightarrow Z$ (cf. \cite{dol}).
\end{dfn}

If $X^{ss}_{\kappa}=Spec A$ then $X^{ss}_{\kappa}//G=Spec A^G$ and 
for quasi-projective manifold $X$, the categorical quotient $X^{ss}_{\kappa}//G$
can be obtained via gluing quotients of affine subsets
(cf. \cite{dol},\cite{thad}).

Dependence of GIT quotient on linearization 
is as follows (cf. \cite{thad},\cite{dolhu}). 
Let $NS^G(X)$ be the $G$-linearized Neron-Severi group i.e. 
the quotient of the group 
$Pic^G(G)$ of $G$-linearizations by algebraic equivalence. 
Firstly, the 
subsets $X^{ss}_{\kappa},X^s_{\kappa}$ of $X$ 
depend only on the class of linearization $\kappa$
in $NS^G$ (cf. (2.1) in \cite{thad}).
It is convenient to view linearizations as 
elements in $NS^G\otimes \QQ$ (referred to as (classes) of 
fractional linearizations).  
An ample linearization $L$ is called $G$-effective if 
$L^n$ has a $G$-invariant section for some $n>0$.
Linearizations $\kappa$ which are $G$-effective 
and such that $X^{ss}_{\kappa} \ne \emptyset$ 
generate a cone $E^G_{\QQ}$ in
the preimage of the ample 
cone  $A_{\QQ} \subset NS(X)\otimes \QQ$ 
for the map $\nu: NS^G_{\QQ}(X) \rightarrow NS(X) \otimes \QQ$.

Secondly, one has the following key result:

\begin{theorem}\label{existencecone}
(cf.\cite{thad},\cite{dolhu})
\begin{enumerate}
\item  The cone $E_{\QQ}^G$ 
is polyhedral i.e. is an intersection in $NS^G_{\QQ}$ of $\nu^{-1}(A_{\QQ})$
with a polyhedron. It is a finite disjoint union of cones
$C_i$ (called chambers if $dim C_i$ is maximal and called cells 
in general) with the following 
property: $X^{ss}_{\kappa}$ is independent of a linearization $\kappa$
as long as $\kappa$ belongs to a fixed cell.

\item If $t \in [-\epsilon,+\epsilon] \subset E^G_{\QQ}$,
where $[-\epsilon,+\epsilon]$
 is viewed as 
a linear variation of lineraizations  such that 
$t=0$ correspond to a point belonging to a wall
(i.e. a codimension one cell), and $X^{ss}(t)$ is 
the semi-stable locus for the  linearization corresponding to $t$,
then there is inclusion $X^{ss}(\epsilon)\subset X^{ss}(0)$
inducing a projective morphism $X^{ss}(\epsilon) \rightarrow X^{ss}(0)$
which is birational if $X^s(0)\ne \emptyset$.
\end{enumerate}
\end{theorem}

\subsection{Phases.}
The main object of this note is 
a class GIT quotients appearing in the following context.

\begin{dfn}\label{phasesdef}
Let $\L$ be the total space of a linearized line bundle $L$
over a smooth quasi-projective manifold $X$ with an action of 
a reductive connected group $G$.
A phase corresponding to quadruple $(X,G,\L,\kappa)$ where 
$\kappa$ is a  linearization 
of $G$-action on $\L$
is the GIT quotient of the total space of linearized 
bundle $\L$ on $X$
i.e. $\L^{ss}//G$.
\begin{enumerate}
\item 
 A phase is called 
Landau Ginzburg, if it is biholomorphic
to $\CC^N/H$ where $H$ a finite subgroup of a torus in $GL_N(\CC)$.
\item
 A phase is called {\it hybrid} if it is biholomorphic to 
an orbifold bundle over a projective orbifold 
with the fiber as in (1).
\item
 A phase $\L^{ss}//G$ is called {\it ``gauged Landau Ginzurg model''} if it is 
biholomorphic to GIT quotient  $X//H\times A$ for a subgroup 
$H\times A$ of $G$ and $A$ (resp. $H$) is a finite 
abelian (resp reductive) group.
\item A phase is called Calabi Yau if it is biholomorphic 
to the total space $\L$ of a line bundle $L$ over a smooth projective manifold 
$X$ such that $c_1(T(\L)\vert_X)=0$ (here $T(\L)$ is the tangent 
bundle of the space $\L$).

\end{enumerate}
\end{dfn}

The following is useful for explicite descriptions of phases
in the case when the bundle $L$ is trivial. 

\begin{prop}\label{gittheorem} Let $G$ be a reductive group acting on 
a smooth quasi-projective variety $X$ and let $\kappa,\psi \in Char G$ be two 
linearizations 
(cf.(\ref{linearizations})) of {\it trivial} line bundle on $X$.
Then the action of $G$ on $\L=X\times \CC$ has the form  
\begin{equation}\label{psiaction}
g(p,x)=(\psi(g)p,gx)
\end{equation}
for $\psi \in \Char(G)$.
As in (\ref{defstability}), let $(X\times \CC)_{\kappa}^{ss},
(X\times \CC)_{\kappa}^s,
(X\times \CC)^{us}_{\kappa}$
denote semi-stable, stable and unstable loci for $G$-action 
and the linearization $\kappa$ of the trivial bundle 
on $X\times \CC$.  

\begin{enumerate}
\item\label{firstitem}
\noindent If $0 \times X^{s}_{\kappa} \subset (X\times \CC)^{ss}_{\kappa}$ 
then $(X \times \CC)^s_{\kappa}/G$ is biholomorphic to a rank one 
orbibundle over the quotient $X^s_{\kappa}/G$.

\item\label{seconditem}
\noindent Suppose that $0 \times X \subseteq (X\times\CC)^{us}$.
Let $H=Ker \psi$, $H_0$ its connected component and $H/H_0$ 
a finite group of connected components of $H$.
If $(0 \times X) \subset (X\times \CC)^{us}$ then 
$(X \times \CC)//G$ is biholomorphic of a quotient of 
$X//H_0$ by action of $H/H_0$.
\end{enumerate}
\end{prop}

\begin{proof} 
First note, that the fibers of maps of geometric 
quotients can be described as 
follows:
\begin{lemma}\label{fibrationquotient}
 Let $f: Y\rightarrow Z$ be a $G$ equivariant 
holomorphic map of quasi-projective varieties for which 
both geometric quotients $Y/G,Z/G$ exist. 
For $z \in Z$ let $G_z$ denotes the stabilizer of 
$z$. Then the fiber at the orbit of $z \in Z$ of 
induced map $f_G: Y/G \rightarrow Z/G$
of the spaces of $G$-orbits
is biholomorphic to $f^{-1}(z)/Stab(z)$.
\end{lemma}

Projection $\pi: X\times \CC \rightarrow X$ is $G$-equivariant 
and it follows from (\ref{defstability}) 
that $\pi$  induces the map $(X \times \CC)^{ss} \rightarrow X^{ss}$. 
Hence Prop.\ref{gittheorem} (\ref{firstitem}) is a consequence of lemma \ref{fibrationquotient}.

In the case Prop.\ref{gittheorem}(\ref{seconditem}) we have: 
\begin{equation}
(X \times \CC)^{ss}/G=
(X \times \CC^*)^{ss}/G=X^{ss}/H=(X//H_0)/(H/H_0)
\end{equation}
\end{proof}

\subsection{Examples of Phases.} (cf. \cite{n=2}).
\begin{example}\label{lgmodelsigmamodel}
{\it LG model 
and $\sigma$-model cf.\cite{n=2}}. Let $G=\CC^*$ acts on $\CC \times \CC^n$ with 
coordinates $(p,x_1,...,x_n,)$ via:
\begin{equation}\label{wittenaction}
     \lambda (p,x_1,...,x_n)=(\lambda^{-n}p,\lambda x_1,...,\lambda x_n)  
\end{equation}
(i.e. we consider the case of Prop. \ref{gittheorem}
when $X=\CC^n,G=\CC^*$ and $\psi(\lambda)=\lambda^{-n},
\lambda\in \CC^*$).  
Let $\H_p \subset \CC \times \CC^n$ denotes the hyperplane $p=0$.
For linearization $\kappa(\lambda)=\lambda^k, k<0$, 
one has $\H_p=(\CC \times \CC^n)^{us}$, the subgroup $H$ of $G=\CC^*$
from (\ref{gittheorem}) part (2)
is $\Ker \psi=\mu_n$ 
where $\mu_n$ is the group of roots of unity of degree $n$.
Therefore $\CC^n\times \CC//\CC^*=\CC^n/H=\CC^n/\mu_n$.

If $k>0$, then $\H_p\setminus (0,....,0) \subset (\CC \times \CC^n)^s$
(indeed, in this case we have ${\CC \times \CC^n}^{us}=
\{(p,x_1,...,x_n) \vert x_1=...=x_n=0\}$)
and from (\ref{gittheorem}) part (1), $(\CC \times \CC^n)//\CC^*$ 
is the total space of the line bundle $\O(-n)_{\PP^{n-1}}$ 
(cf. Corollary \ref{lgcy} below
for relation of invariants of 
this GIT to invariants of hypersurfaces in $\PP^{n-1}$ which explains the term 
$\sigma$-model phase).
\end{example}  

\begin{example}\label{weightedaction}
{\it LG models and weighted projective spaces.} 
Consider action of $G=\CC^*$ on $\CC \times \CC^n$ 
via
\begin{equation}\label{weightedactioneq}
\lambda \cdot(p,x_1,\cdots,x_n)=(\lambda^{-D}p,\cdots,\lambda^{w_i}x_i,...).
\end{equation}
i.e. $X=\CC^n,\psi(\lambda)=\lambda^{-D}$ in notations of 
Prop. \ref{gittheorem}.
There are two GIT quotients of $\CC^{n+1}$ with respect to 
the action (\ref{weightedactioneq})
corresponding to $\CC^*$-linearizations $\kappa(\lambda)=\lambda^k$ with 
$k<0$ and $k>0$.
The former is isomorphic to $\CC^n/Ker \psi=\CC^n/\mu_D$
where generator $J_W$ of $\mu_D$ acts as 
\begin{equation}\label{gradingop}
    J_W(\dots,x_i,\dots)=(\dots,e^{2 \pi \ii q_i}x_i,\dots)
\end{equation}
(i.e. via exponential grading operator; here $q_i={w_i \over D}$).
The later linearization yields the line bundle 
over the weighted projective space 
$\PP(w_1,\dots,w_n)$ with $c_1=-Dh$, where $h$ is the positive 
generator of 
$H^2(\PP(w_1,\dots,w_n),\ZZ)$.
\par Note that $\CC^n/\mu_D$ 
admits as a compactification the weighted projective space
$\PP(D,w_1,\dots,w_n)$.   
The latter is the quotient of 
$\PP(1,w_1,\dots,w_n)$ by the action  of $\mu_D$.
\end{example}

\begin{example}\label{hybridexample}
 {\it Hybrid model} (cf. \cite{n=2}
sect. 5.2)
{\it Product of projective spaces.}
Let $G=(\CC^*)^2$ acts on $\CC^{n+m+1}$ via
\begin{equation}(\lambda_1,\lambda_2)(p,x_1,\dots,x_n,y_1,\dots,y_m)=
(\psi(\lambda_1,\lambda_2)p,\lambda_1 x_1,\dots,\lambda_1 x_n,\lambda_2y_1,
\dots,\lambda_2y_m)
\end{equation}

where $\psi(\lambda_1,\lambda_2)=\lambda_1^{-n}\lambda_2^{-m}$.
Elements of $Pic^G(\CC^{n+m})=Char((\CC^*)^2)$ have the form 
$\kappa(\lambda_1,\lambda_2)=\lambda^{r_1}_1\lambda_2^{r_2}$.
Different choice of stability conditions lead to 

\begin{enumerate}
\item ($H_1$) For the cone $r_1<0,r_1m-r_2n<0$ (which correspond 
to the case when unstable locus is the union of $p=0$ and $y_1=....=y_m=0$)
one obtains the 
total space of the orbifold $\CC^n/\mu_n$ bundle over $\PP^{m-1}$, 

\item ($H_2$) For the cone $r_2<0,r_1m-r_2n>0$ (where 
the unstable locus is the union of $p=0$ and $x_1=...=x_n=0$),  
the total space of orbifold $\CC^m/\mu_m$ bundle over $\PP^{n-1}$,

\item ($H_3$) For the cone $r_1>0,r_2>0$ one obtains the
total space of line bundle over $\PP^{n-1} \times \PP^{m-1}$.

These GIT quotients are quotients of smooth quasi-projective 
varieties by finite abelian groups as follows:

$H_1$ (resp. $H_2$) is quotient of the total space of 
$\V_1=(\oplus \O_{\PP^{m-1}}(-m))^n$ (resp. $\V_2=(\oplus \O_{\PP^{n-1}}(-n))^m$) 
by the action 
of $\mu_n \subset (\CC^*)^n=Aut (\V_1)$ (resp $\mu_m$). 
Both support the diagonal action of $T=\CC^* \subset Aut(\V_i)$.
\end{enumerate}
\end{example}

\begin{example}\label{gaugedlg} {\it Gauged LG models}:
\begin{enumerate}
\item {\it Linearization on the wall.} Consider the case $m=n=2$ 
in Example \ref{hybridexample} (i) and linearizations satisfying $r_1=r_2$.
The unstable locus is $p=0$ and hence it follows from (\ref{gittheorem})
that $\CC^2 \times \CC^4//\CC^*=\CC^4//H$ where $H=\CC^*\times \mu_2$.
The action of connected component of identity $H_0\subset H$ 
on $\CC^4$ is given by 
$\lambda(x_1,x_2,y_1,y_2)=(\lambda x_1,\lambda x_2,\lambda^{-1}y_1,
\lambda^{-1}y_2)$ and the action of $\mu_2$ is 
$(\lambda x_1,\lambda x_2,y_1,y_2), \lambda \in \mu_2$.  
The GIT quotient by $H_0$ can be identified with the cone in $\CC^4$
given by $T_1T_4=T_2T_3$ where $T_1=x_1y_1,T_2=x_1y_2,T_3=x_2y_1,T_4=x_2y_2$
(cf. \cite{dol}, Ex.8.6). The action 
of $\mu_2$ on the cone is the diagonal action $\lambda (T_1,..,T_4)=
(\lambda T_1,....,\lambda T_4)$ and this phase 
is biholomorphic to $\mu_2$-quotient of the cone. Here
we have gauged LG model corresponding to $\mu_2$ with the gauge 
group $H_0=\CC^*$.

\item{\it Toric varieties.} Let $X$ be a projective toric variety corresponding 
to a simplicial fan. It has a GIT quotient presentation $\CC^N//{\CC^*}^k$
where $\rho: {\CC^*}^k \rightarrow (\CC^*)^N$ is a homomorphism
into the torus acting diagonally on $\CC^N$
(cf. \cite{cox} \cite{CLS}). For a fixed $\psi \in \Char(\CC^*)^k$ 
one has several phases 
of ${\CC^*}^k$-action 
of $\CC\times \CC^N$ given by $g(p,x)=(\psi(g)p,\rho(g)x),g \in (\CC^*)^k,
p\in \CC, x\in \CC^N$, 
one of which is a line bundle over $X$ (for $\kappa \in \Char \CC^k$ such that  
$\{(p,x) \vert p=0,x \in \CC^{N}\} \subset (\CC\times \CC^N)_{\kappa}^{ss}$ and the rest
are toric varieties which are GIT quotients of $\CC^N$ by the 
action of an algebraic group 
with identity component being the torus $(\CC^*)^{k-1}$
i.e. an example of gauge LG model in the sense of (3) Def. \ref{phasesdef}
(in some cases, as in example \ref{hybridexample}, this is also a 
hybrid model).

\item Let $Mat_{n,m}$ denote the vector space of $n \times m$ matrices 
(with entries in $\CC$). 
Consider the action of $GL_n$ on $(Mat_{n.m})\times \CC \ \ (n<m)$ 
via multiplication on the first factor and via the character 
$\psi(A)=det(A)^{-k}, k >0$ on the second. In this case,
 $Pic^{GL_n}=\ZZ$ and one has two GIT quotients
one of which is the line bundle over the Grassmanian $Gr(n,m)$ and 
another is the quotient of the affine cone of $Gr(n,m)$ 
 by the cyclic group
(i.e. this phase is the gauged (with gauge group $SL_n$) $\mu_m$-LG model).
For $k=m$ the first GIT quotient has trivial first Chern class
(recall that $c_1(Gr(n,m)=m\sigma$ where $\sigma$ is a positive
generator of $H^2(Gr(n,m),\ZZ)$ (cf. \cite{BH}, Sect. 16.2)   
\end{enumerate}
\end{example}

\section{$\CC^*$-action on GIT quotients}\label{actionssection}
Following Proposition gives a sufficient condition for a phase 
$\L//_{\kappa}G$ (cf. def. \ref{phasesdef}) to support 
the $\CC^*$-action induced by $\CC^*$ action on $\L$.

\begin{prop}\label{actionexistence} 
 Let $X,G,L,\L,\kappa$ as in Def. \ref{phasesdef} and such that 
$\kappa$ belongs to the interior of a GIT chamber i.e. 
the interior of a cone of codimension zero in $E^G_{\QQ}$ 
described in Theorem \ref{existencecone}.
Then the $\CC^*$-action on the fibers of projection $\pi: \L \rightarrow X$ 
induces the $\CC^*$-action on $\L//_{\kappa}G$.
\end{prop}

\begin{proof} 
We claim that $x \in \L^{us}_{\kappa}$ iff 
$\pi(x) \in X^{us}_{\kappa}$.
Indeed, let $x \in \L$, $K$ be the line bundle on $X$ underlying the 
linearization $\kappa$ and let us assume that $x$ is $\kappa$-unstable.
Let $s \in \Gamma(X,K^m)^G, m \in \NN$.
Then $\pi^*(s) \in \Gamma(\L,\pi^*(K)^m)^G$ for 
the linearization of $\pi^*(K)$ induced by $\kappa$ 
and hence $\pi^*s(x)=0$, 
i.e. $s(x)=0$. 

Vice versa, let $x \in \L$ and $\pi(x) \in X$ be unstable. Consider 
$s \in \Gamma(\L,(\pi^*(K)^m)^G)$.
From Leray spectral sequence: $H^p(X,R^q\pi_*(\pi^*(K)^m))\Rightarrow
H^{p+q}(\L,\pi^*(K)^m)$,
and vanishing of $R^q\pi_*(\pi^*(K)^m)$ for $q>0$,
it follows that  $$H^0(\L,\pi^*(K)^m)=
H^0(X,\pi_*(\pi^*(K)^m))=H^0(X,K^m\otimes \pi_*(\O_{\L})).$$
Using the decomposition of $\pi_*(\O_{\L})$ into eigensheaves
of $\CC^*$-action on $\L$ and identifying these eigensheaves with the sheaves 
$L^n, n \in \ZZ$ we obtain the isomorphism:
\begin{equation}\label{identification}
\Gamma (\L,(\pi^*(K)^m)^G)=\oplus_n \Gamma(X,K^m\otimes L^n)
\end{equation}
Now, since $\kappa$ is in the interior of a GIT chamber, for $0<\epsilon <<1$ 
and fixed $n$ the linearization 
$\kappa+\epsilon n \psi$, where $\psi$ is the linearization corresponding 
to the $G$-bundle $L$, belongs to the same GIT chamber as $\kappa$.
This implies that for each $n$ and $m>>0$ the linearization  
$m\kappa+n\psi=m(\kappa+{n \over m}\psi)$ belongs 
to the interior of 
the same chamber as $\kappa$ and hence invariant section 
of $K^m \otimes L^n$ must vanish at $x$. Applying this to the finite 
collection on integers $n$ for which the components 
of decomposition (\ref{identification})
of $s$ are non trivial we see that any section
in $H^0(X,\pi_*(\pi^*(K)^m))=H^0(\L,\pi^*(K)^m)$ 
vanishes at $x$ for large $m$ 
i.e. $x$ is unstable.


This implies that $\L^{ss}$ is $\CC^*$-invariant 
and hence restriction induces the $\CC^*$-action on semi-stable locus. 
Since $L$ is a $G$-bundle,
 this $\CC^*$-action commutes 
with $G$-action the gluing construction of 
$\L//_{\kappa}G$ yields that considered $\CC^*$ action 
descends to this GIT quotient.  
\end{proof}






Singularities of the open subset $X^s/G \subset X^{ss}//G$ 
are quotient singularities (corresponding to orbits of points with 
finite non-trivial stabilizers) while singularities in the complement $X^{ss}\setminus X^s$ 
can be more 
complicated (cf. \cite{kirwanpartial} for a description 
of their resolution). 

It follows from \cite{geras} that a quotient $X//G$ 
by the action of a torus can be represented 
as {\it a global quotient} i.e. there exist a smooth quasi-projective variety 
$\widetilde X$ and finite group $\Gamma$ such that 
$X//G=\widetilde {X}/\Gamma$. The next proposition discusses liftings 
of $\CC^*$-actions in the general context of global quotients by finite
groups.

\begin{prop}\label{morphismgit}
Let $\kappa_1$ and $\kappa_2$ be two linearizations 
of  $G$-action on $\L$ as in Def. \ref{phasesdef} and let 
$X_i=(\L)//_{\kappa_i}G$ for $i=1,2$ be corresponding 
GIT quotients such that 
there is a birational morphism $\psi: X_1 \rightarrow X_2$. 
Assume that 
\begin{enumerate}
\item there exist a finite group $\Gamma$ acting 
on a smooth manifold $\tilde X_2$
and morphism $\tilde \psi$ of $\tilde X_1=\tilde X_2 \times_{X_1} X_1$ making the diagram 
commutative: 
\begin{equation}\label{pushout}
\begin{matrix}\tilde X_1 & \buildrel \tilde \psi \over \rightarrow & \tilde X_2
 \cr \pi_1\downarrow &  & \pi_2 \downarrow \cr 
\tilde X_1/\Gamma=X_1 & \buildrel \psi \over 
\rightarrow & \tilde X_2/\Gamma=X_2 \cr
\end{matrix}
\end{equation}



\item the action of $T=\CC^*$ on total space $\L_2$ of bundle
$L_2$ descends to the $\CC^*$-action on 
$X_2$ (e.g. if conditions of Prop. \ref{actionexistence} are met) 
so that it preserves the branching locus of $\pi_2$.
\end{enumerate}

Then this action on $X_2$ 
lifts to the action of a finite 
cover $\tilde T$ of $T$ on 
$\tilde X_2, X_1,\tilde X_1$ so that the diagram (\ref{pushout}) 
is $\tilde T$-equivariant.\footnote{action of $\tilde T$ on $X_1,X_2$ is not effective 
with $\Ker \tilde T \rightarrow T$ acting trivially on $X_1,X_2$}
\end{prop}

\begin{proof} Conditions for the local lifting of automorphisms 
of quotients by a finite group on ramification divisor 
of the quotient from \cite{losik}
are satisfied for the GIT quotients $X_1,X_2$. 
Such a lift is unique up to possible ambiguity due to analytic continuation
which produces a well defined element in the covering group so that 
the action of the covering group $\tilde T$ is well defined.
\end{proof}



\section{Elliptic genus}\label{ellgensection}

\subsection{Orbifold Equivariant elliptic genus of pairs}
Let $\G$ and $\Gamma$ be respectively an algebraic group and a finite 
group both acting on quasi-projective manifolds $X$ 
via biholomorphic automorphisms so that two actions commute.
The following describes expansion of holomorphic Euler characteristic
of a sheaf which is both $\G$ and $\Gamma$ equivariant
in terms of the characters of $\G$.
Note that our choice of generators 
in the lattice of characters of $\G$ inside $\Char \G \otimes \QQ$
depends on $\Gamma$ and made so that the lattice of characters of 
the quotient of $\G$ acting on $X/\Gamma$ 
effectively will be primitive. For the rest of the paper 
we will be interested in the case $\G=\CC^*$ so we shall denote 
the algebraic group of the automorphisms as $T$.

\begin{dfn}\label{equivareulerchar} 
Let $X$ be a smooth quasi-projective variety,
$T$ and $\Gamma$ be a connected algebraic
 and finite abelian groups respectively 
both acting holomorphically on $X$ so that both actions commute. 
Let $\Gamma_0=T \cap \Gamma$ be normal in $T$ and $T_0=T/\Gamma_0$. 
Let $\F$ be a sheaf on $X$ which is 
 both  $T$ and $\Gamma$-equivariant. 
Assume that the actions on $\F$ commute so each $\Gamma$-eigensheaf
${\F}_{\lambda}$ of $\F$ supports the action of $T$.
Equivariant Euler characteristic of ${\F}_{\lambda}$ is 
\begin{equation}
       \chi^T(\F_{\lambda})=\sum (-1)^idim H^i(X,\F_{\lambda})_me^m=
\int_X ch_T(\F_{\lambda})Td_T(X) 
\end{equation}
where $m \in Char(T_0) \otimes \QQ =Char(T)\otimes \QQ$ 
and $e^m$ is corresponding to a character $m \in Char T_0$ 
element of group ring of $Char T_0$, $ch_T, Td_T$ are the 
$T$-equivariant
Chern and Todd classes respectively 
(cf. \cite{EG} for a discussion of 
equivariant Riemann-Roch).
\end{dfn}
Elements of $\Char T$ are 
linear combinations with $\QQ$-coefficients of elements in $\Char T_0$
and the series (\ref{equivareulerchar}) has fractional exponents.  

Though $T_0$ does not act on $X$ (but does so on $X/\Gamma$) 
the reason to express $\chi_T(\F)$ in terms of characters of $T_0$
is that we will be interested in sheaves $\F$ used to describe 
the invariants of $X/\Gamma$. 
In particular we will consider equivariant version of elliptic class 
introduced in \cite{BLAnnals}. We refer to this paper 
for details of the definitions including the one for $T$-normal pair 
and for numerous applications. See 
 \cite{waelder} for discussion of equivariant case.

\begin{dfn}\label{ellgendef} 
Let $(X,E)$ be a Kawamata log-terminal $\Gamma$-normal pair (in particular,
$X$ is smooth and $E$ has simple normal crossings)
with $E=-\sum_k\delta_kE_k$. Assume that $T$ and $\Gamma$ are
 as in definition \ref{equivareulerchar}
and that the action of $\Gamma$ on the set of 
 components $E_k$ of $E$ is trivial.
For a character $\lambda$ of the 
subgroup $(g,h)$ of $\Gamma$ generated by a pair 
of commuting elements $g,h \in \Gamma$, denote by $V_{\lambda}$ 
the $\lambda$-eigen-bundle of the restriction of $TX\vert_{X^{g,h}}$ 
on the set $X^{g,h}$ of fixed points of $g$ and $h$. 
Let $0 \le \lambda(g)<1$ be the logarithm of the value of character on $g$ 
(i.e. the value of character at $g$ is $e^{2 \pi i \lambda(g)}$). 
Let $x_{\lambda}(t) \in H^*_{T}(X^{g,h})$  be equivariant Chern roots 
of $V_{\lambda}$, $e_k(t) \in H^2_{T}(X)$ be the equivariant 
Chern classes of $T$-bundles
$\O(E_k)$ and let 
$0 \le \epsilon_k<1 $ denotes the logarithm of the 
character of $(g,h) \subset \Gamma$ 
acting on the bundle 
$\O(E_k)$ if $X^{g,h} \subset E_k$ and zero otherwise.  
Finally let $i_{X^{g,h}*}: H^*_{T}(X^{g,h}) \rightarrow H^*_{T}(X)$
be the Gysin and $\Psi: X \rightarrow X/\Gamma$ be the quotient map.

The \emph{orbifold elliptic class} of the triple $(X,E,\Gamma)$ is 
an element 
\newline $\Psi_*(\Ell_{orb}^T(X,E,\Gamma;u,z,\tau)) \in H^*_{T_0}(X/\Gamma)$
where  
\begin{equation}\label{equivargenus}
{\cal Ell}_{orb}^T(X,E,\Gamma;u,z,\tau)):=
\frac 1{\vert \Gamma\vert }\sum_{g,h,gh=hg}\sum_{X^{g,h}}(i_{X^{g,h}})_*
\Bigl(
\prod_{\lambda(g)=\lambda(h)=0} x_{\lambda}(t) 
\Bigr)
$$
$$\times\prod_{\lambda} \frac{ \theta(\frac{x_{\lambda}(t)}{2 \pi \ii }+
 \lambda (g)-\tau \lambda(h)-z )} 
{ \theta(\frac{x_{\lambda}(t)}{2 \pi \ii }+
 \lambda (g)-\tau \lambda(h))}  \ee^{2 \pi \ii \lambda(h)z}
$$
$$\times\prod_{k}
\frac
{\theta(\frac {e_k(t)}{2\pi\ii}+\epsilon_k(g)-\epsilon_k(h)\tau-(\delta_k+1)z)}
{\theta(\frac {e_k(t)}{2\pi\ii}+\epsilon_k(g)-\epsilon_k(h)\tau-z)}
{}
\frac{\theta(-z)}{\theta(-(\delta_k+1)z)} \ee^{2\pi\ii\delta_k\epsilon_k(h)z}.
\end{equation}
\end{dfn}
\noindent ($T_0$ is as in Def. \ref{equivareulerchar})

Localization theorem represents equivariant euler characteristic 
as a sum over fixed components. For example, for a torus
  $T$ acting on a manifold $X$
one has (cf. \cite{EG}): 
\begin{equation}\label{localization}
\chi^T(X,\F)=\sum_{P \in X^T} \chi^T(X,P,i_{Y}^*(\F))
\end{equation} 
where summation is over connected components $P$ 
of the fixed point set $X^G$ of $G$ and 
$i^*$ is the pull back map of equivariant cohomology.

In the case of (\ref{equivargenus}), in the term 
of the localization form 
(\ref{localization})
corresponding to a component $P$ of the fixed point set 
of the $G$-action,  
the class $e_k(t)$ get replaced 
by the class of equivariant line bundle associated with the 
divisor $E_k$ restricted 
to $P$. 
\begin{dfn}
The equivariant elliptic genus ${Ell}_{orb}^T(X,E,\Gamma;u,z,\tau)$
 ($u \in Lie(T)^*\otimes \CC$)
is the value 
of class  ${\Ell}_{orb}^T(X,E,\Gamma;u,z,\tau)$ on the 
fundamental class of $X$. If $P$ is a fixed irreducible component 
of $T$-action then contribution of $P$ is the 
value the localization of elliptic class ${\Ell}_{orb}^T(X,E,\Gamma;u,z,\tau))$ 
on the fundamental class of $P$.
\end{dfn}

\subsection{Elliptic genus of phases}

Now we apply the set up of the last section to the case 
of GIT quotients. This is done with assumptions of existence 
of the structure of global orbifold
on the GIT quotient. 
The definition will be stated only in the case when $E$ in Def.\ref{ellgendef}
is empty.

\begin{dfn}\label{defellgenphase} ({\it Elliptic genus of a phase}) Let 
$X,G,L,\L,\kappa$ 
be as in Def. \ref{phasesdef}. 
Assume that this data satisfies the conditions of both Prop.
\ref{actionexistence} and \ref{morphismgit}.
In particular $\L//_{\kappa}G$ is endowed with the $T=\CC^*$-action, 
there exist 
a smooth manifold $\widetilde {\L//_{\kappa}G}$ acted upon by a finite abelian 
group $\Gamma$  and by $\CC^*$ so that these actions
commute and such that $\widetilde {\L//_{\kappa}G}/\Gamma=
\L//_{\kappa}G$.  
Let $P \subset \widetilde{\L//_{\kappa}G}$ be an 
irreducible component fixed by the $\CC^*$-action 
on $\widetilde {\L//_{\kappa}G}$. 
Consider the equivariant  
elliptic class 
$\Ell^{\CC^*}_{orb}(\widetilde {\L//_{\kappa}G},\Gamma,u,z,\tau)$ 
where $u$ is the infinitesimal character of $\CC^*$ acting faithfully 
on orbifold $\L//_{\kappa}G$.
Then the elliptic genus of the phase $(X,G,L,\L,\kappa)$ 
relative to the component $P$, denoted as
$Ell(\L//_{\kappa}G,P,z,\tau)$, is defined as the 
restriction of the equivariant 
elliptic genus $Ell^{\CC^*}_{orb}(\widetilde {\L//_{\kappa}G},\Gamma,P,u,z,\tau)$
on the diagonal $u=z$ of $\CC\times \CC\times \H$:
\begin{equation}Ell(\L_{\kappa}//G,P,z,\tau)=
Ell^{\CC^*}_{orb}(\widetilde {\L//_{\kappa}G},\Gamma,P,z,z,\tau)
\end{equation} 
More generally, the same definition will be used in the 
cases when $\widetilde {\L//_{\kappa}G}$ has Kawamata 
log-terminal singularities and when $Ell(\widetilde {\L//_{\kappa}G},\Gamma)$
is well defined as the orbifold elliptic 
genus of a pair via a resolution of singularities and taking 
into account the divisor determined by the discrepancies of the resolution
(cf. \cite{BLAnnals}).
\end{dfn}

We shall give examples in the next section but first we  
discuss modularity properties of our elliptic genus.
For a discussion of modular properties of elliptic genus of manifolds 
see \cite{invent}.

\subsection{Modularity}\label{modularitysec} 
\begin{theorem}\label{modularitytheorem}
 Let $\L$ be the total space of a $G$-equivariant 
line bundle over a quasi-projective manifold $X$ and 
$\X=\L//_{\kappa}G$ be the phase corresponding to a linearization $\kappa$
of $G$-action on $\L$. Assume that $\X$ 
admits a presentation 
as a global quotient $\X=\widetilde {\X}/\Gamma$ as in Prop. \ref{morphismgit}
and, in addition, that the orbifoldization group $\Gamma$ is 
subgroup of the lift of $\CC^*$-action on $\widetilde \X$ as also 
described in \ref{morphismgit}. 

Let $P \subset \widetilde {\X}$ be a compact component of the fixed point 
 set 
for this $\CC^*$-action,  
$T_{\widetilde \X}\vert_P$ be the tangent bundle to $\widetilde \X$ restricted 
on $P$ and $c_1^{eq}(T_{\widetilde \X}\vert_P) \in H^2_{\CC^*}(P)$
be its equivariant 
first Chern class.
If

\bigskip
(a) $c_1(T_{\widetilde \X}\vert_P)=0$, 

\bigskip

(b) the Calabi Yau condition (\ref{cy}) described below 
is met, 
\bigskip
\newline \noindent then the restriction of the $\CC^*$-equivariant elliptic genus 
$Ell^{\CC^*}_{orb}(P,\widetilde {\X},\Gamma,u,z)$ on $u=z$ is a Jacobi 
form of weight zero and index ${d \over 2}={{\dim \widetilde {\X}} \over 2}-1$.
\end{theorem}

\begin{proof} Let $n$ be the order of group $\Gamma$, which 
is cyclic since we assume that $\Gamma \subset \CC^*$. 
We shall identify it with the group of roots of unity $\mu_n$. 
Let $T_{\widetilde \X}\vert_P=\oplus V_i$ be a split into a 
sum of rank one eigen-bundles of the above $\CC^*$-action
(some possibly equivariantly isomorphic). 
Let $x_i \in H^2(P,\ZZ)$ be the first Chern class of $V_i$
and $q_i \in \QQ$ is such that 
$q_iu$ be the infinitesimal character of $\CC^*$-action 
corresponding to the bundle $V_i$
($u$ is the infinitesimal character of the group $\CC^*/\Gamma$ acting 
on $\widetilde \X/\Gamma$). Note that $nq_i \in \ZZ$ since ${\rm ord}(\Gamma)=n$.
For $\chi_i \in \Char \mu_n$ corresponding to the action of   
$\Gamma=\mu_n$ on summand $V_i$ we have: 
$\chi_i(1)=exp {{2 \pi i q_i}}$. 
Vanishing of equivariant first Chern class 
implies that:
\begin{equation}
 \sum_i x_i=0, 
\end{equation}
We shall assume further the following Calabi Yau condition: 
\begin{equation}\label{cy}
   \sum_{i=1}^{i=n} q_i=1
\end{equation}
This implies that 
$$(c) \ \ \Gamma \ {\rm acts \ trivially \ on} \ det T_{\tilde X}\vert_P$$
since we assume that $\Gamma \subset \CC^*$. 
As in \cite{sing} (and Def. \ref{ellgendef} above), 
we replace $\chi \in \Char \mu_n$ by logarithm 
$\lambda: \mu_n \rightarrow \QQ$ such that 
$\chi(g)=exp(2 \pi i \lambda(g)), 0 \le \lambda(g) <n$.
With these notations, one has the following expression for 
the contribution into the orbifold elliptic genus of $\X$ corresponding 
to a component $P$ (similar to expression 
for orbifold elliptic genus in \cite{sing}):
\begin{equation}\label{ellgenphase}
Ell(\L//_{\kappa}G,P)=Ell^{\CC^*}_{orb}(\widetilde {\X},P,\mu_n)=
\end{equation}
$${1\over n}\sum_{g,h, \in \mu_m} \prod_{\lambda(g)=\lambda(h)=0}x_{\lambda}
\prod_{i}\Phi(g,h,\lambda_i,x_i,z,\tau)[\widetilde {\X}^{g,h}]
$$
The factor $\Phi(g,h,\lambda,x,z,\tau)$
defined by:
\begin{equation}
\Phi(g,h,\lambda,x,z,\tau)=
{{\theta({x \over {2 \pi i}}+(q_{\lambda}-1)z+\lambda(g)-\lambda(h)\tau)}
\over {\theta({x \over {2 \pi i}}+q_{\lambda}z+\lambda(g)-\lambda(h)\tau)}}
e^{2 \pi i \lambda(h) z}
\end{equation}
where $q_{\lambda}\in \QQ$ is such that 
$q_{\lambda}u$ is the infinitesimal character of 
the $\CC^*$-action on the eigen-bundle corresponding to $\lambda$ 
and $x \in H^2(P,\ZZ)$. This is specialization to $u=z$ 
of the equivariant version of the corresponding expression 
in \cite{sing}.
For $a,b \in \ZZ/n\ZZ, q \in {1 \over n}\ZZ$ we define 
\begin{equation}
\Psi(a,b,q,x,z,\tau)=
{{\theta({x \over {2 \pi i}}+(q-1)z+{{qa-qb\tau}\over n})}
\over {\theta({x \over {2 \pi i}}+qz+{{qa-qb\tau}\over n})}}
e^{2 \pi i z{qb\over n}}  
\end{equation}
so that for $g,h \in \ZZ/n\ZZ, g=a \bmod n,h=b \bmod n$ and $q=q_{\lambda} $ 
one has: 
$\Phi(g,h,\lambda,x,z,\tau)= \Psi(a,b,q,x,z,\tau)$.
We have the following identities:
\begin{equation}\label{id1}
 \Psi(a-1,b,x,q,z+1,\tau)=-\Psi(a,b,x,z,\tau)
\end{equation}
\begin{equation}\label{id2}
\Psi(a,b+1,x,q,z+\tau,\tau)=e^{x-2 \pi i (1-2q)z+2\pi i qa-(1-2q)\pi i \tau} 
\Psi(a,b,x,z,\tau)
\end{equation}
\begin{equation}\label{id3}
\Psi(a-b,b,x,q,z,\tau+1)=\Psi(a,b,x,z,\tau)
\end{equation}
\begin{equation}\label{id4}
\Psi(a,b,x,q,{z \over \tau},-{1 \over \tau})=
e^{-xz+{{\pi i (1-2q)z^2} \over {\tau}}-2 \pi i zq a}
\Psi(b,-a,x,z,\tau)
\end{equation}
Identity (\ref{id1}) uses $\theta(z+1,\tau)=-\theta(z,\tau)$ 
(cf. \cite{chandr}). The identity (\ref{id3}) is clear and 
(\ref{id4}) follows as the corresponding identity in the proof 
of the theorem 4.3 in \cite{sing}.
Finally 
$$\Psi(a,b+1,x,z+\tau,\tau)=
{{\theta({x \over {2 \pi i}}+(q_i-1)(z+\tau)+
qa-q(b+1)\tau)}
\over {\theta({x \over {2 \pi i}}+q(z+\tau)+qa-q(b+1)\tau)}}
e^{2 \pi i q(b+1) (z+\tau)}=
$$ 
$${{\theta({x \over {2 \pi i}}+(q-1)z-\tau+
qa-qb\tau) }
\over {\theta({x \over {2 \pi i}}+qz+qa-qb\tau)}}
e^{2 \pi i qb z}e^{2 \pi i q z}e^{2 \pi i qb \tau}e^{2 \pi i q \tau}
$$
Using $\theta(z-\tau)=-\theta(z)e^{2 \pi i z-\pi i \tau}$ we obtain (\ref{id2}).
Now Jacobi property of (\ref{ellgenphase}) follows since $\sum_i(1-2q_i)=
\dim \X-2$.
\end{proof}

\subsection{LG elliptic genus}
Another kind of elliptic genus, which is associated with 
weighted homogeneous polynomials, was proposed in 
physics literature (cf. \cite{kawai},\cite{berglund} 
and references therein).

\begin{dfn}\label{berglundgenus}
(LG elliptic genus) (cf. \cite{berglund}).
 Let $G_W$ be abelian group and $R$ a representation of $G$ in $\CC^n$ 
which preserves a weighted
homogeneous polynomial given in coordinates of a basis 
$e_i, i=1,...,n$ as $W(x_1,...,x_n)$. Let $w_i,D \in \NN, i=1,...,n$
(weights and degree of $W$)
be integers such that $W(...,\lambda^{w_i}x_i,....)=\lambda^DW(...,x_i,...)$
and $q_i={w_i \over D} \in \QQ$. 
Assume that in basis $e_i$ the group $G$ acts 
diagonally and via $R(g) \cdot e_i=R_i(g)e_i=exp (2 \pi \ii \theta_i(g)) e_i, 
\theta_i(g) \in \QQ$. Finally, let $H$ be a subgroup of $G$.
Then the elliptic genus $Z[R,H]$ of the the data $(G,R,W,H)$ is given by 
\begin{equation}\label{formulaberglund}
    Z[R,H]={ 1\over {\vert H\vert}}\sum_{h_a,h_b\in H}\Pi_{i=1}^{i=n} 
Z[R_i](h_a,h_b)
\end{equation}
where 
\begin{equation} 
    Z[R_i](h_a,h_b)=e^{-2 \pi \ii \theta_i(h_a)}
{{\theta((1-q_i)z+\theta_i(h_b)-\tau \theta_i(h_a),\tau)}\\
\over{\theta(q_iz+\theta_i(h_b)+\tau \theta_i(h_a),\tau)}}
\end{equation}
\end{dfn}

The following follows by direct calculations and description 
of spectrum of weighted homogeneous singularities obtained in \cite{steenbrink}
(cf. also \cite{ebeling}; recall that the exponential grading operator $J_W$ 
is given by (\ref{gradingop})).

\begin{theorem}Specialization of 
LG elliptic genus for $\tau=\ii \infty, t=exp(2 \pi \ii z)$ coincides 
with the orbifoldization of generating function 
$\sum_{\alpha \in \QQ} dim H_{exp(2 \pi \ii \alpha)}t^{\alpha}$ where 
$\alpha$ runs through the spectrum of isolated singularity 
 $W=0$ and $H_{\chi}$ is the eigenspace of the $Gr^{[\alpha]}_FH^{n-1}(M_W)$
graded vector space of the Hodge filtration of Milnor fiber of $W$.
Orbifoldization group $H$ is the 
cyclic group generated by the exponential grading operator $J_W$ (cf. \cite{ebeling}).
\end{theorem}

Now for Landau Ginzburg phase (cf. Example \ref{lgmodelsigmamodel}) the only fixed point
of the $\CC^*$-action on $\CC^n/<J_W>$ is the origin. The infinitesimal 
characters of the lift of this action on $\CC^n$ are ${w_iu}\over D$ 
(where $u$ is the infinitesimal character of the $\CC^*$ action on 
$\CC^n/<J_W>$).
Hence we obtain:

\begin{prop} The elliptic genus of Landau Ginzburg phase (cf .
Def \ref{phasesdef},(2)) 
coincides with the LG elliptic genus (\ref{berglundgenus}).
\end{prop}

\section{Main theorem and Explicit forms of LG/CY 
correspondence}\label{maintheoremsection}

\subsection{Equivariant elliptic genus in birational morphisms}

Recall the following equivariant version of McKay correspondence 
for elliptic genus (theorem 
5.3 in \cite{BLAnnals} and theorem 10 in \cite{waelder}).
We give a weaker form which assumes existence of crepant resolution
since this is sufficient in all examples we consider below. 
We refer to \cite{BLAnnals} and \cite{waelder} for versions which involves 
elliptic genus of pairs and includes corrections corresponding to 
discrepancies of a resolution map. 
\begin{theorem}\label{mckay} (local equivariant McKay correspondence) 
Let $X$ be a smooth quasi-projective 
variety, $\Gamma$ and $G$ are respectively finite and reductive 
groups acting on $X$ so that the actions commute. Let $T$ be the maximal 
torus of $G$.
Assume that 
\begin{enumerate}
\item there exist a crepant $T$-equivariant resolution 
$\pi: \  \tilde X \rightarrow 
X/\Gamma$ i.e. $K_{\tilde X}=\pi^*(K_{X/\Gamma})$.
\item fixed point sets of $T$ action on $X$ and $\tilde X$ are compact.
\end{enumerate}
Then
\begin{equation}\label{mckayequation}
\Ell_{orb}^T(X,\Gamma,P)=\sum_{P_i}\Ell^T(\tilde X,P_i)
\end{equation}
where the sum is taken over all 
fixed components $P_i$ of $T$-action mapped to $P$.
\end{theorem}


\begin{remark}\label{liealgebrafunction} Equivariant 
elliptic genera appearing in 
(\ref{mckayequation}) are functions on 
$\CC \times (Lie(T)^*\otimes \CC) \times \H$
where the first and third coordinates correspond to 
the variables of theta function $\theta(z,\tau)$ and $Lie(T)^*$
is the dual of the Lie algebra of the maximal torus
i.e. the space of infinitesimal characters of representation of $T$.
Recall again that as was done in Def. \ref{defellgenphase} and Theorem 
\ref{modularitytheorem}, 
we normalize variables in $Lie(T)^* \otimes \CC$ 
by choosing basis given by the characters of the quotient group 
of $T$ which acts faithfully on $X/\Gamma$ i.e. the group 
$T/T \cap \Gamma$. 
\end{remark}




\subsection{Main theorem}

\begin{theorem}\label{main}
 Let $\L//_{\kappa_1}G=X_1=\bar X_1/\Gamma,\L//_{\kappa_2}G=
X_2=\bar X_2/\Gamma,
\tilde X_1,\tilde X_2,\Gamma$ are as 
in Prop. \ref{morphismgit}. Assume that 
$\psi: X_1 \rightarrow X_2$ is 
a K-equivalence i.e. $\psi^*(K_{X_2})=K_{X_1}$.
Then
\begin{equation}\label{maintheoremformula}
  \sum_{P_i}Ell(\L//_{\kappa_1},P_i)=
 Ell(\L//_{\kappa_2},P)
\end{equation}
where $P_i$ is collection of fixed point sets which $\psi$ takes 
into $P$.
\end{theorem}

\begin{proof} 
McKay correspondence (cf. theorem \ref{mckay}) asserts that 
the orbifold elliptic genus, in particular yielding the elliptic 
genus of the phase cf. Def. \ref{defellgenphase}),
coincides with the singular elliptic genus of the quotient.
The elliptic genus of the quotient, as the elliptic 
genus of a singular variety is given in terms of resolution
of its singularities. Our assumption that 
phases $\L//_{\kappa_1}G$ and $\L//_{\kappa_2}G$ are related 
by $K$-equivalence imply that both expressions in terms of 
resolution are the same: cf. Prop. 3.7 in \cite{sing}.
 Hence the assertion follows.
\end{proof}
The reset of this section considers explicit forms of the identity 
(\ref{maintheoremformula}) for the examples of phases discussed
by Witten in \cite{n=2}.

\subsection{Projective space}\label{projectivespace}

Consider the  GIT quotient $\CC^n\times \CC/\CC^*$  
as in example \ref{lgmodelsigmamodel} (1), i.e. for the 
action $\lambda (p,x_1,...,x_n)=(\lambda^{-n}p,
\lambda x_1,...,\lambda x_n)$. 
Assume first that $k<0$ for  linearization $\kappa(\lambda)=\lambda^k$.
In this case GIT quotient is the orbifold 
$X_1=\CC^n/\mu_n$ (cf. \ref{lgmodelsigmamodel}). 
The $\CC^*$-action from Prop. \ref{actionexistence}
is the action of 1-dimensional torus $\CC^*/\mu_n$ induced 
on $X_1$ by the diagonal action of $\CC^*$ on $\CC^n$. 
The contribution of the origin 
i.e. the only fixed point of this 
action 
(for the action 
of $\CC^*$ expressed in terms of 
generator $u \in \Char \CC^*/\mu_n)$)  
is given by 
\begin{equation}\label{ordinarylg}
     ({{\theta ({u \over n}-z)} \over {\theta({u \over n})}})^n
\end{equation}
Orbifoldized elliptic genus (\ref{ordinarylg}) 
has form
(cf. (\ref{equivargenus})):
\begin{equation}
   \Ell_{orb}^{\CC^*}(\CC^n,P,\mu_n)= \sum_{a,b}(e^{-2 \pi \ii {b z\over n}}
{{\theta(({u \over n}-z+{{a-b\tau}\over n},\tau)}\\
\over{\theta({u \over n}+{{a-b\tau} \over n},\tau)}})^n
\end{equation}
For $u=z$ we obtain the LG-genus (\ref{formulaberglund}) where $
q_i={1 \over n}$:
\begin{equation}\label{lgequalpowers}
={1 \over n}\sum_{0 \le a,b <n}({{\theta(({{1-n}\over n})z+
{{a-b\tau} \over n})}
\over {\theta({z \over n}+{{a-b \tau} \over n})}}
\ee^{2 \pi \ii {{bz} \over n}})^n
\end{equation}

Another GIT quotient corresponds to linearizations with 
$k>0$ in which case unstable locus is the line $x_1=...=x_N=0$
(cf. \ref{lgmodelsigmamodel}(1)).
It is biholomorphic to the total space $X_2$ of the 
bundle $\O_{\PP^{N-1}}(-N)$.
If $\tilde X_2$ is the blow up of $\CC^n$ at the origin, 
one has $\tilde X_2/\mu_n=X_2$. Let $P$ be the exceptional $\PP^{n-1}$.
Then the equivariant Chern class satisfies: 
$c(TP)=(1+x)^n,c(N_p)=-nx+2\pi \ii u)$. Therefore the contribution of the 
exceptional set $P$, which is also the fixed point set of 
$\CC^*$ action, in the equivariant 
elliptic genus of resolution is given by:
\begin{equation}\label{equivaregprojspace} 
  \Ell^{\CC^*}(X_2,P)=
({x{\theta ({x \over {2 \pi \ii}}-z)}\over {\theta ({x \over 
{2 \pi \ii}})}})^{n}
{{\theta(-{nx \over {2 \pi \ii}}+{u}-z)}
\over {\theta(-{nx \over {2 \pi \ii}}
+{u},\tau)}}
 \end{equation}

\begin{corollary}\label{lgcy}
 (LG-CY correspondence, cf. \cite{GM}) LG elliptic genus 
of singularity $x_1^n+.....+x_n^n$ coincides with the 
elliptic genus of smooth CY hypersurface in $\PP^{n-1}$.
\end{corollary}

\begin{proof} It follows from equivariant McKay correspondence
(theorem \ref{mckay}) that
\begin{equation}\label{lgcyequivarformula}
  Ell^{\CC^*}_{orb}(\CC^n,\mu_n,\O)=Ell^{\CC^*}(\widetilde{\CC^n/\mu_n},P)
\end{equation}
Hence LG genus (\ref{lgequalpowers}), i.e. the left hand side of 
(\ref{lgcyequivarformula}) for $u=z$ 
is the elliptic genus 
(\ref{equivaregprojspace})  restricted to $u=z$ i.e. can be written as 
\begin{equation}\label{ellgencyprojspace}
({x{\theta ({x \over {2 \pi \ii}}-z)}\over {\theta ({x \over 
{2 \pi \ii}})}})^{n}
{{\theta({nx \over {2 \pi \ii}})}
\over {\theta({nx \over {2 \pi \ii}}
-z,\tau)}}[\PP^{n-1}]
\end{equation}
Since the total Chern class of smooth hypersurface $V_n$ of degree $n$ 
in $\PP^{n-1}$
is $c(V_n)={{(1+x)^n}\over {(1+nx)}}$
($x \in H^2(\PP^{n-1},\ZZ)$
 is the positive generator) 
the expression (\ref{ellgencyprojspace}) conicides with 
the elliptic genus 
of $V_n$.
\end{proof}

\subsection{Weighted $\CC^*$-actions.}\label{weigthedactions}

Now we consider the extension of the previous case, 
to the action with unequal weights considered in 
example \ref{weightedaction}.

\begin{theorem}
Let $T=\CC^*$ acts on $\CC^{n+1}$ 
via (\ref{weightedactioneq})
and let 
\begin{equation}\label{cyweighted}
D=\sum w_i, q_i={w_i \over D}
\end{equation} i.e. the Calabi Yau condition (\ref{cy})  
is satisfied. 
\begin{enumerate}
\item\label{firstpartofprop} The $\CC^*$-action discussed in Prop. \ref{morphismgit} on $\tilde X_1=\CC^n$ which is the lift of 
the $\CC^*$-action on $X_1=\CC^n/\mu_D=\CC^n \times \CC//T$ 
(where the latter induced from the action
of $\CC^*$ on $\CC^n \times \CC$ given by 
$s(p,z_1,....,z_n)=(s \cdot p,z_1,...,z_n)$)
has the form: $s(x_1,....,x_n)=(s^{w_1}x_1,....,s^{w_n}x_n)$.
The action on the second quotient $X_2$, i.e. the line bundle
over $\PP(w_1,...,w_n)$ is the fiberwise action of $\CC^*$.
The contraction of the zero section of the line bundle $X_2$ 
is biholomorphic to $X_1$.
\item\label{localcontributionlg}
 The restriction of local contribution of the origin $\O$ of $X_1$
into equivariant orbifold elliptic genus 
$Ell^{\CC^*}(\CC^n,<J_W>,\O)$ 
obtained by restriction 
on subset  of variables given by $u=z$  
coincides with the LG elliptic genus corresponding to the data
$(G_W,R,W,<J_w>)$ (cf. def. \ref{berglundgenus})
where 
$W$ is a weighted homogeneous polynomial with weights $w_i$ and 
degree $D$,
$G_W$ generated by $(....,exp (2 \pi \ii {w_i \over D},....)$
and $<J_W>$ is cyclic subgroup of $G_W$.       
generated by exponential grading operator (\ref{gradingop}).

\item\label{caselocalconribweighted}
 The local contribution for $\CC^*$-action on  
the total space of the line bundle over $\PP^{n-1}(w_1,...,w_n)$ in 
Example \ref{weightedaction}
is given by 
\begin{equation}\label{localcontribweighted}
\Ell_{orb}(\PP^{n-1},\mu_{w_1}\times \cdots \times \mu_{w_n}) \cdot 
{{\theta({Dx \over {2 \pi \ii}}+u-z)}
\over {\theta({Dx \over {2 \pi \ii}}
-z,\tau)}}[\PP^{n-1}(w_1,....,w_n)]
\end{equation}
where 
$\Ell_{orb}(\PP^{n-1},\mu_{w_1}\times \cdots \times \mu_{w_n})$ is the total 
orbifold elliptic class of $\PP^{n-1}(w_1,...,w_n)$ considered 
as the orbifold quotient of $\PP^{n-1}$. 
\item (LG-CY correspondence) 
If the hypersurface $V_D$ of degree $D$ in $\PP^{n-1}({w_1,..,w_n})$
is quasi-smooth and with assumption (\ref{cyweighted})
the elliptic genus (\ref{berglundgenus}) coincides with the 
orbifold elliptic genus of hypersurface $V_D$.
\end{enumerate}
\end{theorem}
\begin{proof} Part (\ref{firstpartofprop}) follows from definitions.
Next consider the local contribution described in 
(\ref{localcontributionlg}). 
The map of groups of characters $j^*: Char \CC^*/\mu_D \rightarrow Char 
\CC^*$ satisfies:
$j^*({u})=Dv$. 
To determine the contribution of the origin
$\O$, which is the only fixed component of the 
$\CC^*$-action, in $Ell^{\CC^*}_{orb}(\CC^n/\mu_n)$ 
note that that 
the infinitesimal characters of action of $\CC^*$  
are $uw_i$.  
Hence the local contribution of $\O$ in terms of generator $u$ 
of  $\Char \CC^*/\mu_d$
is given as follows (where $\lambda_i(g)$ as in (\ref{ellgendef})):
\begin{equation}\label{annalsspecial} 
{ 1 \over {D}}\sum_{g,h \in \mu_D}\Pi_{i=1}^{i=n} 
{{\theta({{u w_i} \over D}-\lambda_i(g)-\lambda_i(h) \tau-z)}
\over {\theta({{uw_i} \over D}-\lambda_i(g)-\lambda_i(h) \tau)}}
\ee^{2 \pi \ii \lambda_i(h)z}
\end{equation}
$$
={ 1 \over {\vert D \vert}}\sum_{0 \le a,b <D}\Pi_{i=1}^{i=n} {{\theta({uq_i}+
{{a-b\tau} \over D}-z)}
\over {\theta({u q_i}+{{a-b \tau} \over D})}}
\ee^{2 \pi \ii {{bz} \over D}}
$$
Specializing this to the case $u=z$ yields the expression
\footnote{
(\ref{annalsspecial})
is the Jacobi form of weight zero and index $n-2$
due to equality $\sum q_i=1$ cf. Theorem \ref{modularitytheorem}.}
$$
={ 1 \over {\vert D \vert}}\sum_{0 \le a,b <D}\Pi_{i=1}^{i=n} {{\theta(({q_i}-1)z+
{{a-b\tau} \over D})}
\over {\theta({zq_i}+{{a-b \tau} \over D})}}
\ee^{2 \pi \ii {{bz} \over D}}
$$
identical to (\ref{formulaberglund})
\footnote{in terminology of \cite{berglund},
$\mu_D$ is the group of phase symmetries.}

In the case (\ref{caselocalconribweighted}) the argument is 
similar to the used to derive (\ref{equivaregprojspace})
but we use partial resolution in which the the only component of 
the exceptional 
set is the orbifold (i.e. $\PP^{n-1}(w_1,....w_n)$). The formula
for the contirbution follows from presentation of 
this orbiofld as global quotient which results in 
replacing the total elliptic class by the orbifold elliptic class.

The last statement follows from results of 
\cite{BLAnnals},\cite{waelder} due 
to relation between the orbifold elliptic class 
and the elliptic class of resolutions.
\end{proof}

\subsection{Hybrid models} Consider now elliptic genus 
for the phase in example \ref{hybridexample}

\begin{theorem} 
The local contributions of the fixed point set of $\V_1$ which 
is the zero section of $\V_1$ 
for the above action of $T$ (i.e. the elliptic genus of phase $H_1$)
is given by:
\begin{equation}\label{hybridellgen}
      {1 \over n} \sum_{0 \le a,b<n} 
{{\theta(-m{x \over {2 \pi i }}+({1 \over n}-1)z+{{a-b\tau}\over n},\tau)}
\over{\theta(-m{x \over {2 \pi i}}+{1 \over n}z+{{a-b\tau}\over n}
,\tau)}}e^{{2 \pi i bz}\over n})^n
({x{\theta({x \over {2 \pi i}}-z)} \over {\theta({x \over {2 \pi i }})}})^m[\PP^{m-1}]
\end{equation}
(where $x=c_1(\O_{\PP^{m-1}}(1))$).
Similarly the elliptic genus of phase $H_2$ 
is given by 
\begin{equation}
      {1 \over m} \sum_{0 \le a,b<m} 
{{\theta(-{nx \over {2 \pi i }}+({1 \over m}-1)z+{{a-b\tau}\over m},\tau)}
\over{\theta(-{nx \over {2 \pi i}}+{1 \over m}z+{{a-b\tau}\over m}
,\tau)}}e^{{2 \pi i bz}\over m})^m
({x{\theta({x \over {2 \pi i}}-z)} \over {\theta({x \over {2 \pi i }})}})^n[\PP^{n-1}]
\end{equation}
Both these contributions are equal to the contribution of 
the phase $H_3$ (``hybrid-CY correspondence''). 
All 3 contribution (up to a factor) coincide with the 
elliptic genus of CY hypersurface in $\PP^{n-1}\times \PP^{m-1}$
(of bi-degree $(n,m)$).
\end{theorem}

\begin{proof} The GIT quotient in the case ($H_1$) is the global quotient 
by the group $\Gamma=\mu_n$ given by 
$\V_1/\mu_n=[\oplus \O_{\PP^{m-1}}(-m)]/\mu_n$ (where $[F]$ denotes 
the total space of the bundle  $F$). $\CC^*$-action is the diagonal action 
on the fibers of $\V_1$ i.e. the fixed point set is $\PP^{m-1}$. 
The equivariant Chern class of the restriction of the tangent bundle to $\V_1$ 
on $\PP^{m-1}$ which is the zero section is 
$c(\V_1,P)=(-mx+{u \over n})^n(1+x)^m$. The eigenbundles of action 
of $\Gamma$ are the same as for $\CC^*$ action so that $\lambda(g)$
for $g=\exp({{2\pi i a}\over n})$ is ${a \over n}$.
Hence the formula (\ref{equivargenus}) for equivariant elliptic genus 
reduces to  
\begin{equation}\label{hybridellgenequivar}
      {1 \over n} \sum_{0 \le a,b<n} 
{{\theta(-m{x \over {2 \pi i }}+{1 \over n}u-z+{{a-b\tau}\over n},\tau)}
\over{\theta(-m{x \over {2 \pi i}}+{1 \over n}u+{{a-b\tau}\over n}
,\tau)}}e^{{2 \pi i bz}\over n})^n
({x{\theta({x \over {2 \pi i}}-z)} \over {\theta({x \over {2 \pi i }})}})^m[\PP^{m-1}]
\end{equation}
and (\ref{hybridellgen}) follows.
The case of ($H_2$) phase is identical. The last assertion 
on equality of \ref{hybridellgenequivar} and the elliptic 
genus of CY hypersurface in $\PP^{m-1}\times \PP^{n-1}$ follows since 
blow up of $\V_1/\mu_n$ is a K-equivalence.
\end{proof}

This example can be easily extended to the case of weighted actions
of $(\CC^*)^2$ on $\CC^n\times\CC^m\times \CC$ given by:
\begin{equation}
(\lambda_1,\lambda_2)(p,x_1,\dots,x_n,y_1,\dots,y_m)=
\end{equation}
$$(\lambda_1^{-D_1}\lambda_2^{-D_2}p,\lambda_1^{w_1^1} x_1,\dots,
\lambda_1^{w^1_n} x_n,\lambda_2^{w^2_1}y_1,
\dots,\lambda_2^{w^2_m}y_m)
$$
One has three phases, one of which is a line bundle over product 
of weighted projective spaces $\PP(w^1_1,\cdots,w^1_n) \times 
\PP(w^2_1,\cdots,w^2_m)$ and others are orbi-bundles over 
each of the factors in this product.
The first two phases are the global quotients.
One of the orbi-bundles is
the quotient of 
$\V=(\O_{\PP^{m-1}}(-D_1))^n$ by the 
orbifold group $\Gamma=\mu_{D_2}\times (\mu_{w_1^1} \times \cdots \times 
\mu_{w_n^1})$  with $\mu_{D_1}$ acting
by multiplying $i$-th summand by $\exp({{2 \pi i w_i^1}\over {D_1}})$ and 
the group $\mu_{w_1^1} \times \cdots \times 
\mu_{w_n^1}$ acts coordinate-wise on $\PP^{n-1}$ so that the quotient 
is $\PP(w^1_1,\cdots,w^1_n)$.
The elliptic genus of this phase is 
given by the formula which is a modification 
of expression (\ref{hybridellgen}) 
as follows. The group $\mu_n$ is replaced by 
the group generated by exponential 
grading operator and factor $({x{\theta({x \over {2 \pi i}}-z)} \over {\theta({x \over {2 \pi i }})}})^m$ by the cohomology class 
which is the class of the orbifold 
$\Ell(\PP^{m-1},\mu_{w_1^2}\cdots,\mu_{w_m^2})$ (cf. \cite{sing}).

\subsection{Gauged Landau Ginzburg models.} Consider now 
the elliptic 
genera in the cases discussed
in Example \ref{gaugedlg}. In  
Example \ref{gaugedlg} (1), 
the GIT quotient is the global quotient by group $\Gamma=\mu_2$ 
acting on a singular space which is the quadratic cone 
in $\CC^4$. It can be resolve either via 
a small resolution or with exceptional set $\PP^1 \times \PP^1$.
If $X$ is either of these resolutions, then 
$Ell^{\CC^*}_{orb}(X,\Gamma,z,z,\tau)$  
yields the elliptic genus of this gauged LG model i.e. 
we obtain several expression for the elliptic genus of such phase. 

Similarly in the case of Example \ref{gaugedlg} (3) 
the elliptic genus of this GIT quotient 
is the equivariant orbifold elliptic (specialized to $u=z$) 
of the resolution of the affine cone over the 
Grassmanian with the orbifold group $\mu_n$. 

\section{Concluding remarks.}

\begin{remark}{\it Elliptic genus of phases when CY condition fails.}
The assumptions of the theorem \ref{main} can be weakened. 
Firstly, the condition that $\psi$ is a K-equivalence can be
eliminated by using equivariant elliptic genus of pairs 
described in Definition \ref{ellgendef} and relating elliptic 
genus of the phase $\L//_{\kappa_2}$ to elliptic genus of 
a pair $(\L//_{\kappa_1},E)$. 

Secondly, assumption of Prop. \ref{morphismgit} that $X_1,X_2$ 
can be replace by requiring that $\bar X_1,\bar X_2$ are
Kawamata log-terminal. 

Finally one can extend Prop. \ref{morphismgit} to Kawatmata log-terminal 
pairs and show \ref{main} 
replacing phases $\L//_{\kappa_i}$ by
pairs $(\L//_{\kappa_i},E_i)$  where $E_i$ are $\CC^*$ invariant divisors
and the latter pairs are klt.
\end{remark}

\begin{remark}{\it Mirror symmetry for hybrid and gauged LG models.}
Existence of LG/CY correspondence suggests that mirror symmetry between
Calabi Yau manifolds for which LG/CY correspondence defined, 
should correspond to mirror symmetry between LG models.
Such construction was proposed in \cite{BeH} and was studied
in detail more recently in \cite{Kra},\cite{ChRBergl},\cite{borisovbergl}.
Correspondence described here between other classes of phases
suggests that there should be a mirror correspondence
between certain hybrid models and more general phases discussed
above. It would be interesting to have such mirror correspondence 
between hybrid models explicitly 
extending Begrlund-Hubsch mirror symmetry for weighted homogeneous polynomials. 
\end{remark}

\begin{remark}{\it Limit $q \rightarrow 0$}

Such specialization yields Hodge theoretical data for sigma-models
(i..e the Hirzebruch's $\chi_y$ genus, cf. \cite{sing}) 
and for Landau-Ginzburg models where it involve 
information about the spectrum of weighted homogeneous singularity 
(cf.\cite{berglund},
and section \ref{weigthedactions} above). It would be interesting 
to have Hodge theoretical interpretation of such limit 
for hybrid and gauged LG models
as well.
\end{remark}

\begin{remark}{\it Possible generalizations} 
In recent paper \cite{borisovkaufmann} the authors
considered generalizations of LG-CY correspondence to (0,2)-models on 
the level of vertex algebras (case of ordinary LG/CY on the 
level of vertex algebras 
in the homogeneous case was considered in \cite{GM}).
It would be interesting to find ``vertex algebras of phases''
extending results of these works.
\end{remark}

\begin{remark} As was mentioned (cf. remark \ref{liealgebrafunction}), 
equivariant orbifold elliptic genus 
is a holomorphic function 
on $\CC \times Lie(T)^*\otimes \CC \times \H$ while the elliptic genus 
of a phase, 
which in Calabi Yau case is a Jacobi form on $\CC\times \H$,
is obtained by restricting this function 
on the line $u_i=z, i=1,...,\dim T$ 
in $\CC \times Lie(T)^*\otimes \CC$ 
(though we were concerned
only with the case $dim T=1$ one can extend construction
of this paper to the case $dim T>1$ as well). It would be interesting 
to describe modular properties of these functions
on $\CC \times Lie(T)^*\otimes \CC \times \H$ and characterize 
mentioned restriction in modular terms.
\end{remark}

\end{document}